\documentclass{article}

\usepackage{ccfonts}
\usepackage{amssymb,amsmath,amsthm}
\usepackage{mathrsfs}
\usepackage{epsfig}
\usepackage[usenames,dvipsnames]{color}

\newcommand\ie{{\em i.e.}}
\newcommand\eg{{\em e.g. }}

\def\B{\mathscr B}
\def\C{\mathbb C}
\def\d{\mathrm{d}}

\def\H{\mathcal H}

\def\L{\mathscr L}

\def\N{\mathbb N}

\def\R{\mathbb R}

\def\U{\mathscr U}

\def\dom{\mathcal D}

\def\div{\mathop{\mathrm{div}}\nolimits}
\def\e{\mathop{\mathrm{e}}\nolimits}

\def\re{\mathop{\mathsf{Re}}\nolimits}
\def\linf{\mathsf{L}^{\:\!\!\infty}}
\def\ltwo{\mathsf{L}^{\:\!\!2}}

\def\slim{\mathop{\hbox{\rm s-}\lim}\nolimits}

\newtheorem{Theorem}{Theorem}[section]

\newtheorem{Lemma}[Theorem]{Lemma}
\newtheorem{Assumption}[Theorem]{Assumption}

\begin{document}


\title{Spectral analysis of time changes for horocycle flows}

\author{R. Tiedra de Aldecoa$^1$\footnote{Supported by the Fondecyt Grant 1090008 and
by the Iniciativa Cientifica Milenio ICM P07-027-F ``Mathematical Theory of Quantum
and Classical Magnetic Systems'' from the Ministerio de Econom\'ia, Fomento y
Turismo.}
}
\date{\small}
\maketitle \vspace{-1cm}

\begin{quote}
\emph{
\begin{itemize}
\item[$^1$] Facultad de Matem\'aticas, Pontificia Universidad Cat\'olica de Chile,\\
Av. Vicu\~na Mackenna 4860, Santiago, Chile
\item[] \emph{E-mail\;\!:} rtiedra@mat.puc.cl
\end{itemize}
}
\end{quote}


\begin{abstract}
We prove (under the condition of A. G. Kushnirenko) that all time changes for the
horocycle flow have purely absolutely continuous spectrum in the orthocomplement of
the constant functions. This provides an answer to a question of A. Katok and J.-P.
Thouvenot on the spectral nature of time changes for horocycle flows. Our proofs rely
on positive commutator methods for self-adjoint operators.
\end{abstract}

\textbf{2000 Mathematics Subject Classification\;\!:} 37C10, 37D40, 81Q10, 58J51.

\smallskip

\textbf{Keywords\;\!:} Horocycle flow, time change, spectral analysis, commutator
methods.

\section{Introduction}
\setcounter{equation}{0}

The purpose of this note is to provide an answer to a question of A. Katok and J.-P.
Thouvenot on the spectral nature of time changes for horocycle flows.

The set-up is the standard one. Consider the unit tangent bundle $M:=T^1\Sigma$ of a
compact Riemann surface $\Sigma$ of genus $\ge2$. The $3$-manifold $M$ carries a
probability measure $\mu$ which is preserved by two distinguished one-parameter
groups of diffeomorphisms\;\!: the horocycle flow $\{F_{1,t}\}_{t\in\R}$ and the
geodesic flow $\{F_{2,t}\}_{t\in\R}$. One associates to those flows vector fields
$X_j$, Lie derivatives $\L_{X_j}$ and unitary groups $\{U_j(t)\}_{t\in\R}$ in
$\ltwo(M,\mu)$ in the usual way. It is a classical result that the horocycle flow
$\{F_{1,t}\}_{t\in\R}$ is uniquely ergodic \cite{Fur73} and mixing of all orders
\cite{Mar78}, and that the unitary group $\{U_1(t)\}_{t\in\R}$ has countable Lebesgue
spectrum \cite{Par53}. Furthermore, A. G. Kushnirenko \cite[Thm.~2]{Kus74} has proved
that all time changes of the horocycle flow are strongly mixing under a condition
which holds if the time change is sufficiently small in the $C^1$ topology. Namely,
if $f\in C^\infty(M)$ satisfies $f>0$ and $f-\L_{X_2}(f)>0$, then the flow of the
vector field $fX_1$ is strongly mixing. This implies that the unitary group
associated to $fX_1$ has purely continuous spectrum, except at $1$, where it has a
simple eigenvalue.

Nothing more is known about the spectral properties of the time change $fX_1$ (see
the comments in \cite[Sec.~1]{AGU11} and \cite[Sec.~1]{Kus74}). However, as pointed
out by A. Katok and J.-P. Thouvenot in \cite[Sec.~6.3.1]{KT06}, it looks plausible
that the unitary group associated to $fX_1$ has purely absolutely continuous or
Lebesgue spectrum. In fact, A. Katok and J.-P. Thouvenot state as a conjecture the
stability of the countable Lebesgue spectrum (see \cite[Conject.~6.8]{KT06}). In the
present note, we give an answer to the first interrogation of these authors by
proving that the unitary group associated to $fX_1$ has purely absolutely continuous
spectrum outside $\{1\}$ under the condition of A. G. Kushnirenko.

Our proof relies on a refined version \cite{ABG,Sah97_2} of a commutator method
introduced by E. Mourre \cite{Mou81}. It uses as a starting point the well-known
commutation relation satisfied by the unitary groups of the horocycle flow and the
geodesic flow\;\!:
\begin{equation}\label{abs_com}
U_2(s)\;\!U_1(t)\;\!U_2(-s)=U_1(\e^st),\quad s,t\in\R.
\end{equation}
To some extent, this approach has been suggested to us by the proof of A. G.
Kushnirenko itself, since it already took advantage of commutator identities linking
the vector fields $X_1,X_2$ and $fX_1$. We also aknowledge the influence of the
article \cite{FRT12} on commutator methods for unitary operators, and we refer to
\cite{AGU11,Fay01,Fay06,FKW01,GM10,Hum74,Tot66} for related works on ergodic and
spectral properties of time changes. In the future, we hope that commutators methods
could be used to derive spectral properties of other classes of flows than the
horocycle flows considered here.

Here is a brief description of the note. In Section \ref{Sec_com}, we recall some
definitions and results on positive commutator methods for self-adjoint operators. In
Section \ref{Sec_abstract}, we introduce a generalisation of the setting presented
above\;\!: We consider on an abstract (possibly noncompact) $n$-manifold vector
fields $X_1,X_2$ and flows $\{F_{1,t}\}_{t\in\R},\{F_{2,t}\}_{t\in\R}$ with unitary
groups satisfying \eqref{abs_com}. Under an assumption generalising the one of A. G.
Kushnirenko (see Assumption \ref{ass_f}) we show that the self-adjoint operator
associated to the time change $fX_1$ has purely absolutely continuous spectrum,
except at $\;\!0$, where it may have an eigenvalue (see Theorem \ref{thm_spec}). We
use the theory of Section \ref{Sec_com} to prove this result. In Section
\ref{Sec_horo} we apply this abstract result to the horocycle flow, taking into
account the fact that the horocycle flow is strongly mixing under the condition A. G.
Kushnirenko. This leads to the desired result, namely, that the unitary group
associated a time change of the horocycle flow has purely absolutely continuous
spectrum outside $\{1\}$ (see Theorem \ref{main_theorem}).

\section{Positive commutator methods}\label{Sec_com}
\setcounter{equation}{0}

We recall in this section some facts on positive commutator methods borrowed from
\cite{ABG} and \cite{Sah97_2} (see also the original paper \cite{Mou81} of E.
Mourre). Let $\H$ be a Hilbert space with norm $\|\;\!\cdot\;\!\|_\H$ and scalar
product $\langle\;\!\cdot\;\!,\;\!\cdot\;\!\rangle_\H$, and denote by $\B(\H)$ the
set of bounded linear operators on $\H$. Let also $A$ be a self-adjoint operator in
$\H$ with domain $\dom(A)$, and $S\in\B(\H)$. For any $k\in\N$, we say that $S$
belongs to $C^k(A)$, with notation $S\in C^k(A)$, if the map
\begin{equation}\label{eq_group}
\R\ni t\mapsto\e^{-itA}S\e^{itA}\in\B(\H)
\end{equation}
is strongly of class $C^k$. In the case $k=1$, one has $S\in C^1(A)$ if the quadratic
form
$$
\dom(A)\ni\varphi\mapsto\big\langle\varphi,iSA\varphi\big\rangle_\H
-\big\langle A\varphi,iS\varphi\big\rangle_\H\in\C
$$
is continuous for the topology induced by $\H$ on $\dom(A)$. We denote by $[iS,A]$
the bounded operator associated with the continuous extension of this form, or
equivalently the strong derivative of the function \eqref{eq_group} at $t=0$.

If $H$ is a self-adjoint operator in $\H$ with domain $\dom(H)$ and spectrum
$\sigma(H)$, we say that $H$ is of class $C^k(A)$ if $(H-z)^{-1}\in C^k(A)$ for some
$z\in\C\setminus\sigma(H)$. If $H$ is of class $C^1(A)$, then the quadratic form
$$
\dom(A)\ni\varphi\mapsto\big\langle\varphi,(H-z)^{-1}A\;\!\varphi\big\rangle_\H
-\big\langle A\;\!\varphi,(H-z)^{-1}\varphi\big\rangle_\H\in\C
$$
extends continuously to a bounded form defined by the operator
$\big[(H-z)^{-1},A\big]\in\B(\H)$. Furthermore, the set $\dom(H)\cap\dom(A)$ is a
core for $H$ and the quadratic form
$$
\dom(H)\cap\dom(A)\ni\varphi\mapsto\big\langle H\varphi,A\varphi\big\rangle_\H
-\big\langle A\varphi,H\varphi\big\rangle_\H\in\C
$$
is continuous in the topology of $\dom(H)$ \cite[Thm.~6.2.10(b)]{ABG}. This form
extends uniquely to a continuous quadratic form on $\dom(H)$ which can be identified
with a continuous operator $[H,A]$ from $\dom(H)$ to the adjoint space $\dom(H)^*$.
In addition, the following relation holds in $\B(\H):$
\begin{equation}\label{2com}
\big[(H-z)^{-1},A\big]=-(H-z)^{-1}[H,A](H-z)^{-1}.
\end{equation}

Let $E^H(\;\!\cdot\;\!)$ denote the spectral measure of the self-adjoint operator
$H$, and assume that $H$ is of class $C^1(A)$. Then, the operator
$E^H(J)\big[iH,A\big]E^H(J)$ is bounded and self-adjoint for each bounded Borel set
$J\subset\R$. If there exists a number $a>0$ such that
$$
E^H(J)\big[iH,A\big]E^H(J)\ge aE^H(J),
$$
then one says that $H$ satisfies a strict Mourre estimate on $J$. The main
consequence of such an estimate is to imply a limiting absorption principle for $H$
on $J$ if $H$ is also of class $C^2(A)$. This in turns implies that $H$ has no
singular spectrum in $J$. We recall here a version of this result valid even if $H$
has no spectral gap (see \cite[Sec.~7.1.2]{ABG} and \cite[Thm.~0.1]{Sah97_2} for the
most general version of this result)\;\!:

\begin{Theorem}\label{thm_abs}
Let $H$ and $A$ be self-ajoint operators in a Hilbert space $\H$. Suppose that $H$ is
of class $C^2(A)$ and satisfies a strict Mourre estimate on a bounded Borel set
$J\subset\R$. Then, $H$ has no singular spectrum in $J$.
\end{Theorem}

\section{Spectral analysis of time changes for abstract flows}\label{Sec_abstract}
\setcounter{equation}{0}

Let $M$ be a $C^\infty$ manifold of dimension $n\ge1$ with volume form $\Omega$, and
let $\{F_{j,t}\}_{t\in\R}$, $j=1,2$, be (nontrivial) $C^\infty$ complete flows on $M$
preserving the measure $\mu_\Omega$ induced by $\Omega$. Then, it is known that the
operators
$$
U_j(t)\;\!\varphi:=\varphi\circ F_{j,t},\quad\varphi\in C^\infty_{\rm c}(M),
$$
define strongly continuous unitary groups $\{U_j(t)\}_{t\in\R}$ in the Hilbert space
$\H:=\ltwo(M,\mu_\Omega)$ (here $C^\infty_{\rm c}(M)$ stands for the space of
$C^\infty$ functions with compact support in $M$). Since $C^\infty_{\rm c}(M)$ is
dense in $\H$ and left invariant by $\{U_j(t)\}_{t\in\R}$, it follows from Nelson's
theorem \cite[Prop.~5.3]{Amr09} that the generator of the group $\{U_j(t)\}_{t\in\R}$
$$
H_j\varphi:=\slim_{t\to0}it^{-1}\big\{U_j(t)-1\big\}\varphi,
\quad\varphi\in\dom(H_j):=\left\{\varphi\in\H
\mid\lim_{t\to0}|t|^{-1}\big\|\big\{U_j(t)-1\big\}\varphi\big\|_\H<\infty\right\}
$$
is essentially self-adjoint on $C^\infty_{\rm c}(M)$. In fact, a direct calculation
shows that
$$
H_j\varphi:=-i\;\!\L_{X_j}\varphi,\quad\varphi\in C^\infty_{\rm c}(M),
$$
where $X_j$ is the (divergence-free) vector field associated to
$\{F_{j,t}\}_{t\in\R}$ and $\L_{X_j}$ the corresponding Lie derivative. Now, suppose
that there exists a $C^1$ isomorphism $e:(\R,+)\to\big((0,\infty),\;\!\cdot\;\!\big)$
such that
\begin{equation}\label{rel_com}
U_2(s)\;\!U_1(t)\;\!U_2(-s)=U_1\big(e(s)\;\!t\big)\quad\hbox{for all }s,t\in\R.
\end{equation}
Then, for each $t\neq0$, $U_1(t)$ has homogeneous Lebesgue spectrum (that is, the
spectrum $\sigma(H_1)$ of $H_1$ covers $\R$, and $\sigma(H_1)\setminus\{0\}$ is
purely Lebesgue with uniform multiplicity, see \cite[Prop.~1.23]{KT06}). Furthermore,
if $\mu_\Omega(M)<\infty$, then any constant function on $M$ is an eigenvector of
$U_1(t)$ with eigenvalue $1$ (in some cases, as when the system
$(M,\mu_\Omega,F_{1,t})$ is ergodic, $1$ is even a simple eigenvalue of $U_1(t)$). By
applying the strong derivative $i\hspace{1pt}\d/\d t$ at $t=0$ in \eqref{rel_com},
one gets that $U_2(s)H_1U_2(-s)\varphi=e(s)H_1\varphi$ for each
$\varphi\in C^\infty_{\rm c}(M)$. Since $C^\infty_{\rm c}(M)$ is a core for $H_1$,
one infers that $H_1$ is $H_2$-homogeneous in the sense of \cite{BG91}; namely,
\begin{equation}\label{eq_homo1}
U_2(s)H_1U_2(-s)=e(s)H_1\quad\hbox{on}\quad\dom(H_1).
\end{equation}
It follows that $H_1$ is of class $C^\infty(H_2)$ with
\begin{equation}\label{com_H1H2}
\big[iH_1,H_2\big]=e'(0)H_1.
\end{equation}

Now, consider a vector field with the same orientation and colinear to the vector
field $X_1$, that is, a vector field $fX_1$ where $f\in C^\infty(M)$ satisfies
$f\ge\delta_f$ for some $\delta_f>0$ and $f\in\linf(M)$. The vector field $fX_1$ has
the same integral curves as $X_1$, but with reparametrised time coordinate. Indeed,
it is known (see \cite[Sec.~1]{Hum74} and \cite[Chap.~2.2]{CFS82} in the compact
case) that the formula
$$
t=\int_0^{h(p,t)}\frac{\d s}{f\big(F_{1,s}(p)\big)}\;\!,\quad(p,t)\in M\times\R,
$$
defines for each $p\in M$ a strictly increasing function $\R\ni t\mapsto h(p,t)\in\R$
satisfying $h(p,0)=0$ and $\lim_{t\to\pm\infty}h(p,t)=\pm\infty$. Furthermore, the
implicit function theorem implies that the map $t\mapsto h(p,t)$ is $C^\infty$ with
$\frac\d{\d t}h(p,t)=f\big(F_{1,h(p,t)}(p)\big)$. Therefore, the function
$\R\ni t\mapsto\widetilde F_{1,t}(p)\in M$ given by
$\widetilde F_{1,t}(p):=F_{1,h(p,t)}(p)$ satisfies the initial value problem 
$$
\frac\d{\d t}\;\!\widetilde F_1(p,t)
=(fX_1)_{\widetilde F_1(p,t)}\;\!,
\quad\widetilde F_1(p,0)=p,
$$
meaning that $\{\widetilde F_{1,t}\}_{t\in\R}$ is the flow of $fX_1$. Since the
divergence $\div_{\Omega/f}(fX_1)$ of $fX_1$ with respect to the volume form
$\Omega/f$ is zero (see \cite[Prop.~2.5.23]{AM78}), it follows by a standard result
\cite[Prop.~2.6.14]{AM78} that the operator
$$
\widetilde H\varphi
:=-i\;\!\L_{fX_1}\varphi
\equiv fH_1\varphi,
\quad\varphi\in C^\infty_{\rm c}(M),
$$
is essentially self-adjoint in $\widetilde\H:=\ltwo(M,\mu_\Omega/f)$. Its closure is
denoted by the same symbol.

In the next lemma, we introduce two auxiliary operators which will be useful for the
spectral analysis of $\widetilde H$.

\begin{Lemma}\label{Lemma_a}
Let $f\in C^\infty(M)$ be such that $f\ge\delta_f$ for some $\delta_f>0$ and
$f\in\linf(M)$. Then,
\begin{enumerate}
\item[(a)] The operator
$$
\U:\H\to\widetilde\H,\quad\varphi\mapsto f^{1/2}\varphi,
$$
is unitary with adjoint $\U^*:\widetilde\H\to\H$ given by $\U^*\psi=f^{-1/2}\psi$.
\item[(b)] The symmetric operator
$$
H\varphi:=f^{1/2}H_1f^{1/2}\varphi,\quad\varphi\in C^\infty_{\rm c}(M),
$$
is essentially self-adjoint in $\H$ (and its closure is denoted by the same symbol).
\item[(c)] For each $z\in\C\setminus\R$, the operator $H_1+zf^{-1}$ is invertible
with bounded inverse, and satisfies
\begin{equation}\label{eq_res}
(H+z)^{-1}=f^{-1/2}\big(H_1+zf^{-1}\big)^{-1}f^{-1/2}.
\end{equation}
\end{enumerate}
\end{Lemma}

\begin{proof}
Point (a) follows from a direct calculation taking into account the boundedness of
$f$ from below and from above. For (b), observe that
$$
H\varphi
=f^{-1/2}fH_1f^{1/2}\varphi
=\U^*\widetilde H\U\varphi
$$
for each $\varphi\in\U^*C^\infty_{\rm c}(M)$. So, $H$ is essentially self-adjoint on
$\U^*C^\infty_{\rm c}(M)\equiv C^\infty_{\rm c}(M)$. To prove (c), take
$z\equiv\lambda+i\mu\in\C\setminus\R$,
$\varphi\in\dom\big(H_1+zf^{-1}\big)\equiv\dom(H_1)$ and
$\{\varphi_n\}\subset C^\infty_{\rm c}(M)$ such that
$\lim_n\|\varphi-\varphi_n\|_{\dom(H_1)}=0$. Then, it follows from (b) that
$$
\big\|\big(H_1+zf^{-1}\big)\varphi\big\|^2_\H
=\lim_n\big\|f^{-1/2}(H+z)f^{-1/2}\varphi_n\big\|^2_\H
\ge\inf_{p\in M}f^{-2}(p)\;\!\mu^2\big\|\varphi\big\|^2_\H,
$$
and thus $H_1+zf^{-1}$ is invertible with bounded inverse (see
\cite[Lemma~3.1]{Amr09}). Now, to show \eqref{eq_res}, take $\psi=(H+z)\zeta$ with
$\zeta\in C^\infty_{\rm c}(M)$, observe that
\begin{equation}\label{eq_res_bis}
(H+z)^{-1}\psi-f^{-1/2}\big(H_1+zf^{-1}\big)^{-1}f^{-1/2}\psi=0,
\end{equation}
and then use the density of $(H+z)C^\infty_{\rm c}(M)$ in $\H$ to extend the identity
\eqref{eq_res_bis} to all of $\H$.
\end{proof}

The proof of Lemma \ref{Lemma_a}(b) implies that $H$ and $\widetilde H$ are unitarily
equivalent. Therefore, one can either work with $H$ in $\H$ or with $\widetilde H$ in
$\widetilde\H$ to determine the spectral properties associated with the time change
$fX_1$. For convenience, we present in the sequel our results for the operator $H$.
We start by collecting all the necessary assumptions on the function $f$ (note that
the assumption $f\in C^\infty(M)$ is made essentially for convenience; if need be,
the results of this note can certainly be extended to the case $f\in C^2(M)$).

\begin{Assumption}[Time change]\label{ass_f}
The function $f\in C^\infty(M)$ is such that
\begin{enumerate}
\item[(i)] $f\ge\delta_f$ for some $\delta_f>0$,
\item[(ii)] the functions $f,\L_{X_1}(f),\L_{X_2}(f)$,
$\L_{X_1}\big(\L_{X_2}(f)\big)$ and $\L_{X_2}\big(\L_{X_2}(f)\big)$ belong to
$\linf(M)$,
\item[(iii)] the function $\displaystyle g:=\frac{e'(0)f-\L_{X_2}(f)}{2f}$ satisfies
$g\ge\delta_g$ for some $\delta_g>0$.
\end{enumerate}
\end{Assumption}

If $M$ is compact, then (ii) is automatically verified and (i) and (iii) are
satisfied if $f$ and $e'(0)f-\L_{X_2}(f)$ are strictly positive functions. Therefore,
Assumption \ref{ass_f} reduces to the assumptions of A. G. Kushnirenko
\cite[Thm.~2]{Kus74}.

In the next lemma, we prove regularity properties of $H$ and $H^2$ with respect
to $H_2$ which will be useful when deriving the strict Mourre estimate.

\begin{Lemma}\label{lem_regu}
Let $f$ satisfy Assumption \ref{ass_f}, take $\alpha\in\{\pm1/2,\pm1\}$ and let
$z\in\C\setminus\R$. Then,
\begin{enumerate}
\item[(a)] the multiplication operators $g^\alpha$ and $f^\alpha$ satisfy
$g^\alpha,f^\alpha\in C^1(H_2)$ and $g^\alpha\in C^1(H)$ with
$$
\big[ig^\alpha,H_2\big]=-\alpha\;\!g^{\alpha-1}\L_{X_2}(g),~~
\big[if^\alpha,H_2\big]=-\alpha f^{\alpha-1}\L_{X_2}(f)~~\hbox{and}~~
\big[ig^\alpha,H\big]=-\alpha fg^{\alpha-1}\L_{X_1}(g),
$$

\item[(b)] $(H+z)^{-1}\in C^1(H_2)$ with
$\big[i(H+z)^{-1},H_2\big]=-(H+z)^{-1}(Hg+gH)(H+z)^{-1}$,

\item[(c)] $\big(H^2+1\big)^{-1}\in C^1(H_2)$ with
$
\big[i\big(H^2+1\big)^{-1},H_2\big]
=-\big(H^2+1\big)^{-1}\big(H^2g+2HgH+gH^2\big)\big(H^2+1\big)^{-1}
$,

\item[(d)] $\big(H^2+1\big)^{-1}\in C^2(H_2)$.
\end{enumerate}
\end{Lemma}

\begin{proof}
(a) Simple computations using the linearity of $\L_{X_2}$ and the
bound $f\ge\delta_f$ imply that
$$
\textstyle\L_{X_2}(f^{1/2})=\frac12\;\!f^{-1/2}\L_{X_2}(f).
$$
Thus, one has for each $\varphi\in C^\infty_{\rm c}(M)$
$$
\textstyle
\big\langle H_2\varphi,f^{1/2}\varphi\big\rangle_\H
-\big\langle\varphi,f^{1/2}H_2\varphi\big\rangle_\H
=\big\langle\varphi,\big[H_2,f^{1/2}\big]\varphi\big\rangle_\H
=\big\langle\varphi,-\frac i2\;\!f^{-1/2}\L_{X_2}(f)\varphi\big\rangle_\H\;\!.
$$
Since $f^{-1/2}\L_{X_2}(f)\in\linf(M)$, it follows by the density of
$C^\infty_{\rm c}(M)$ in $\dom(H_2)$, that $f^{1/2}\in C^1(H_2)$ with
$\big[H_2,f^{1/2}\big]=-\frac i2\;\!f^{-1/2}\L_{X_2}(f)$. The
other identities can be shown similarly.

(b) Let $t\in\R$ and $\varphi\in\H$. Then, one infers from Equations \eqref{eq_homo1}
and \eqref{eq_res} that
\begin{align*}
&\e^{-itH_2}(H+z)^{-1}\e^{itH_2}\varphi\\
&=\e^{-itH_2}f^{-1/2}\e^{itH_2}\big(e(t)H_1+z\e^{-itH_2}f^{-1}\e^{itH_2}\big)^{-1}
\e^{-itH_2}f^{-1/2}\e^{itH_2}\varphi.
\end{align*}
So, one gets from point (a), Equation \eqref{eq_res} and Lemma \ref{Lemma_a}(b) that
\begin{align*}
&\frac\d{\d t}\;\!\e^{-itH_2}(H+z)^{-1}\e^{itH_2}\varphi\Big|_{t=0}\\
&=\big[if^{-1/2},H_2\big]\big(H_1+zf^{-1}\big)^{-1}f^{-1/2}\varphi
+f^{-1/2}\big(H_1+zf^{-1}\big)^{-1}\big[if^{-1/2},H_2\big]\varphi\\
&\qquad-f^{-1/2}\big(H_1+zf^{-1}\big)^{-1}\big\{e'(0)H_1+z\big[if^{-1},H_2\big]\big\}
\big(H_1+zf^{-1}\big)^{-1}f^{-1/2}\varphi\\
&=\frac12\;\!f^{-1}\L_{X_2}(f)(H+z)^{-1}\varphi
+\frac12\;\!(H+z)^{-1}f^{-1}\L_{X_2}(f)\varphi\\
&\qquad-(H+z)^{-1}\big\{e'(0)H+zf^{-1}\L_{X_2}(f)\big\}(H+z)^{-1}\varphi\\
&=\frac12\;\!(H+z)^{-1}Hf^{-1}\L_{X_2}(f)(H+z)^{-1}\varphi
+\frac12\;\!(H+z)^{-1}f^{-1}\L_{X_2}(f)H(H+z)^{-1}\varphi\\
&\qquad-(H+z)^{-1}e'(0)H(H+z)^{-1}\varphi\\
&=-(H+z)^{-1}(Hg+gH)(H+z)^{-1}\varphi,
\end{align*}
which implies the claim.

(c) Let $\varphi\in\H$. Then, it follows from point (b) that
\begin{align*}
&\frac\d{\d t}\;\!\e^{-itH_2}\big(H^2+1\big)^{-1}\e^{itH_2}\varphi\Big|_{t=0}\\
&=\frac\d{\d t}\;\!\e^{-itH_2}(H+i)^{-1}\e^{itH_2}\e^{-itH_2}(H-i)^{-1}\e^{itH_2}
\varphi\Big|_{t=0}\\
&=-(H+i)^{-1}(Hg+gH)(H+i)^{-1}(H-i)^{-1}\varphi
-(H+i)^{-1}(H-i)^{-1}(Hg+gH)(H-i)^{-1}\varphi\\
&=-\big(H^2+1\big)^{-1}\big(H^2g+2HgH+gH^2\big)\big(H^2+1\big)^{-1}\varphi,
\end{align*}
which implies the claim.

(d) Direct computations using point (c) show that
\begin{align*}
\big[i\big(H^2+1\big)^{-1},H_2\big]
&=-\big(H^2+1\big)^{-1}
\big\{\big(H^2+1\big)g-2g+g\big(H^2+1\big)\\
&\hspace{100pt}+2(H+i)g(H-i)-2ig(H-i)+2i(H+i)g\big\}
\big(H^2+1\big)^{-1}\\
&=-2\re\big\{g\big(H^2+1\big)^{-1}-\big(H^2+1\big)^{-1}g\big(H^2+1\big)^{-1}\\
&\hspace{80pt}+(H-i)^{-1}g(H+i)^{-1}+2i(H-i)^{-1}g\big(H^2+1\big)^{-1}\big\}.
\end{align*}
Therefore, the claim readily follows from the fact that the operators
$g,\big(H^2+1\big)^{-1},(H-i)^{-1}$ and $(H+i)^{-1}$ belong to $C^1(H_2)$.
\end{proof}

In order to apply the theory of Section \ref{Sec_com}, one has to prove at some point
a positive commutator estimate. Usually, one proves it for the operator $H$ under
study. But in our case, the commutator $\big[iH,H_2\big]=Hg+gH$ appearing in Lemma
\ref{lem_regu}(b) (which is the simplest nontrivial commutator in our set-up) does
not exhibit any explicit positivity. By contrast, the commutator
$\big[iH^2,H_2\big]=H^2g+2HgH+gH^2$ of Lemma \ref{lem_regu}(c) is made of the
positive operators $g$, $H^2$ and $HgH$, and thus $\big[iH^2,H_2\big]$ is more likely
to be positive as a whole. The formalisation of this intuition is the content of the
next lemma.

\begin{Lemma}[Strict Mourre estimate for $H^2$]\label{M_estimate}
Let $f$ satisfy Assumption \ref{ass_f} and let $J$ be a bounded Borel set in
$(0,\infty)$. Then,
$$
E^{H^2}(J)\big[iH^2,H_2\big]E^{H^2}(J)\ge aE^{H^2}(J)
$$
with $a:=2\;\!\delta_g\cdot\inf(J)>0$.
\end{Lemma}

\begin{proof}
We know from Equation \eqref{2com} and Lemma \ref{lem_regu}(c) that
$$
E^{H^2}(J)\big[iH^2,H_2\big]E^{H^2}(J)=E^{H^2}(J)\big(H^2g+2HgH+gH^2\big)E^{H^2}(J).
$$
We also know from Assumption \ref{ass_f}(iii) that
$$
E^{H^2}(J)2HgHE^{H^2}(J)\ge aE^{H^2}(J)
$$
with $a=2\;\!\delta_g\cdot\inf(J)>0$. Therefore, it is sufficient to show that
$E^{H^2}(J)\big(H^2g+gH^2\big)E^{H^2}(J)\ge0$.

So, for any $\varepsilon>0$ let
$H^2_\varepsilon:=H^2\big(\varepsilon^2H^2+1\big)^{-1}$ and
$H^\pm_\varepsilon:=H(\varepsilon H\pm i)^{-1}$. Then, the inclusion
$g^{1/2}\in C^1(H)$ of Lemma \ref{lem_regu}(a) implies that
$$
\slim_{\varepsilon\searrow0}\big[H^\pm_\varepsilon,g^{1/2}\big]
=\pm\slim_{\varepsilon\searrow0}
(\varepsilon H\pm i)^{-1}\big[iH,g^{1/2}\big](\varepsilon H\pm i)^{-1}\\
=\pm i\big[g^{1/2},H\big].
$$
Therefore, for each $\varphi\in\H$ it follows that
\begin{align*}
&\big\langle\varphi, E^{H^2}(J)\big(H^2g+gH^2\big)E^{H^2}(J)\varphi\big\rangle_\H\\
&=\lim_{\varepsilon\searrow0}\big\langle\varphi,E^{H^2}(J)
\big(H^2_\varepsilon g^{1/2}g^{1/2}+g^{1/2}g^{1/2}H^2_\varepsilon\big)E^{H^2}(J)
\varphi\big\rangle_\H\\
&=\lim_{\varepsilon\searrow0}\big\langle\varphi,
E^{H^2}(J)\big(\big[H^2_\varepsilon,g^{1/2}\big]g^{1/2}
+2\;\!g^{1/2}H^2_\varepsilon g^{1/2}
+g^{1/2}\big[g^{1/2},H^2_\varepsilon\big]\big)E^{H^2}(J)\varphi\big\rangle_\H\\
&\ge\lim_{\varepsilon\searrow0}\big\langle\varphi,
E^{H^2}(J)\big(\big[H^2_\varepsilon,g^{1/2}\big]g^{1/2}
+g^{1/2}\big[g^{1/2},H^2_\varepsilon\big]\big)E^{H^2}(J)\varphi\big\rangle_\H\\
&=\lim_{\varepsilon\searrow0}\big\langle\varphi,
E^{H^2}(J)\big(H^+_\varepsilon\big[H^-_\varepsilon,g^{1/2}\big]g^{1/2}
+\big[H^+_\varepsilon,g^{1/2}\big]H^-_\varepsilon g^{1/2}\\
&\hspace{90pt}+g^{1/2}\big[g^{1/2},H^+_\varepsilon\big]H^-_\varepsilon
+g^{1/2}H^+_\varepsilon\big[g^{1/2},H^-_\varepsilon\big]\big)E^{H^2}(J)
\varphi\big\rangle_\H\\
&=\lim_{\varepsilon\searrow0}\big\langle\varphi,
E^{H^2}(J)\big(H\big[H,g^{1/2}\big]g^{1/2}
+\big[H^+_\varepsilon,g^{1/2}\big]g^{1/2}H^-_\varepsilon
+\big[H^+_\varepsilon,g^{1/2}\big]\big[H^-_\varepsilon,g^{1/2}\big]\\
&\hspace{90pt}+g^{1/2}\big[g^{1/2},H\big]H
+H^+_\varepsilon g^{1/2}\big[g^{1/2},H^-_\varepsilon\big]
+\big[g^{1/2},H^+_\varepsilon\big]\big[g^{1/2},H^-_\varepsilon\big]\big)E^{H^2}(J)
\varphi\big\rangle_\H\\
&=\big\langle\varphi,E^{H^2}(J)\big(H\big[H,g^{1/2}\big]g^{1/2}
+\big[H,g^{1/2}\big]g^{1/2}H
+2\big[H,g^{1/2}\big]^2+g^{1/2}\big[g^{1/2},H\big]H\\
&\hspace{260pt}+Hg^{1/2}\big[g^{1/2},H\big]\big)E^{H^2}(J)\varphi\big\rangle_\H\\
&=\big\langle\varphi,E^{H^2}(J)2\big[H,g^{1/2}\big]^2E^{H^2}(J)
\varphi\big\rangle_\H\\
&\ge0,
\end{align*}
which implies the claim.
\end{proof}

Using the previous results for $H^2$, one can finally determine spectral properties
of $H:$

\begin{Theorem}[Spectral properties of $H$]\label{thm_spec}
Let $f$ satisfy Assumption \ref{ass_f}. Then, $H$ has purely absolutely continuous
spectrum, except at $\;\!0$, where it may have an eigenvalue.
\end{Theorem}

\begin{proof}
We know from Lemmas \ref{lem_regu}(d) and \ref{M_estimate} that
$\big(H^2+1\big)^{-1}\in C^2(H_2)$ and that $H^2$ satisfies a strict Mourre estimate
on each bounded Borel subset of $(0,\infty)$. It follows by Theorem \ref{thm_abs}
that $H^2$ has purely absolutely continuous spectrum, except at $0$, where it may
have an eigenvalue. Accordingly, the Hilbert space $\H$ admits the orthogonal
decomposition
$$
\H=\ker(H^2)\oplus\H_{\rm ac}(H^2),
$$
with $\H_{\rm ac}(H^2)$ the subspace of absolute continuity of $H^2$.

Now, the function $\lambda\mapsto\lambda^2$ has the Luzin N property on $\R$; namely,
if $J$ is a Borel subset of $\R$ with Lebesgue measure zero, then $J^2$ also has
Lebesgue measure zero. It follows that $\H_{\rm ac}(H^2)\subset\H_{\rm ac}(H)$, with
$\H_{\rm ac}(H)$ the subspace of absolute continuity of $H$ (see Proposition 29,
Section 3.5.4 of \cite{BW83}). Since $\ker(A^2)=\ker(A)$ for all self-adjoint
operators $A$, we thus infer that
$$
\H=\ker(H^2)\oplus\H_{\rm ac}(H^2)\subset\ker(H)\oplus\H_{\rm ac}(H).
$$
So, one necessarily has $\H=\ker(H)\oplus\H_{\rm ac}(H)$, meaning that $H$ has purely
absolutely continuous spectrum, except at $\;\!0$, where it may have an eigenvalue.
\end{proof}

\section{Spectral analysis of time changes for horocycle flows}\label{Sec_horo}
\setcounter{equation}{0}

In this section, we apply the results of Section \ref{Sec_abstract} to time changes
for horocycle flows on compact surfaces of constant negative curvature.

Let $\Sigma$ be a compact Riemann surface of genus $\ge2$ and let $M:=T^1\Sigma$ be
the unit tangent bundle of $\Sigma$. The $3$-manifold $M$ carries a probability
measure $\mu_\Omega$ (induced by a canonical volume form $\Omega$) which is preserved
by  two distinguished one-parameter groups of diffeomorphisms\;\!: the horocycle flow
$\{F_{1,t}\}_{t\in\R}$ and the geodesic flow $\{F_{2,t}\}_{t\in\R}$. Both flows
correspond to right translations on $M$ when $M$ is identified with a homogeneous
space $\Gamma\setminus{\sf PSL}(2;\R)$, for some cocompact lattice $\Gamma$ in
${\sf PSL}(2;\R)$ (see \cite[Sec.~III.3 \& Sec.~IV.1]{BM00} for details). We denote
by $\{U_1(t)\}_{t\in\R}$ and $\{U_2(t)\}_{t\in\R}$ the corresponding unitary groups
in $\H:=\ltwo(M,\mu_\Omega)$, and we write $X_j$ (resp. $H_j$) for the vector field
(resp. self-adjoint generator) associated to $\{U_j(t)\}_{t\in\R}$, $j=1,2$ (see
Section \ref{Sec_abstract}). It is a classical result that the horocycle flow
$\{F_{1,t}\}_{t\in\R}$ is uniquely ergodic \cite{Fur73} and mixing of all orders
\cite{Mar78}, and that $U_1(t)$ has countable Lebesgue spectrum for each $t\ne0$ (see
\cite[Prop.~2.2]{KT06} and \cite{Par53}). Moreover, the identity \eqref{rel_com}
holds with $e:\R\to(0,\infty)$ the exponential, \ie
$$
U_2(s)\;\!U_1(t)\;\!U_2(-s)=U_1(\e^st)\quad\hbox{for all }s,t\in\R
$$
(here we consider the negative horocycle flow
$\{F_{1,t}\}_{t\in\R}\equiv\{F_{1,t}^-\}_{t\in\R}$, but everything we say can be
adapted to the positive horocycle flow by inverting a sign, see
\cite[Rem.~IV.1.2]{BM00}).

Now, consider a time change $fX_1$ of $X_1$ with $f\in C^\infty(M)$ satisfying
Assumption \ref{ass_f}, let $H$ be the self-adjoint operator as in Lemma
\ref{Lemma_a}(b), and let $\widetilde H$ be the self-adjoint operator associated to
$fX_1$. Since $M$ is compact, Assumption \ref{ass_f} reduces to the following\;\!:

\begin{Assumption}\label{ass_compact}
The functions $f\in C^\infty(M)$ and $f-\L_{X_2}(f)\in C^\infty(M)$ are strictly
positive.
\end{Assumption}

Under Assumption \ref{ass_compact}, A. G. Kushnirenko \cite[Thm.~2]{Kus74} has shown
that the flow generated by the vector field $fX_1$ is strongly mixing (see
\cite[Sec.~4]{Mar77} for a generalisation of this result). So, $\widetilde H$ has
purely continuous spectrum, except at $0$, where it has a simple eigenvalue (see \eg
\cite[Sec.~I.2]{BM00}). Moreover, the flows
$\{F_{1,t}\}_{t\in\R},\{F_{2,t}\}_{t\in\R}$ and the function $f$ satisfy all the
assumptions of Section \ref{Sec_abstract}. Therefore, Theorem \ref{thm_spec} implies
that $H$ has no singular continuous spectrum. These properties, together with the
fact that $H$ and $\widetilde H$ are unitarily equivalent, lead to the following
result\;\!:

\begin{Theorem}\label{main_theorem}
Let $f$ satisfy Assumption \ref{ass_compact}. Then, the self-adjoint operator
$\widetilde H$ associated to the vector field $fX_1$ has purely absolutely continuous
spectrum, except at $\;\!0$, where it has a simple eigenvalue.
\end{Theorem}


\def\cprime{$'$} \def\cprime{$'$}


\end{document}